\documentclass[11pt,a4paper]{article}
\usepackage{graphicx}
\usepackage{a4wide}
\usepackage{amsmath,amssymb,amsthm,bm}
\usepackage{enumerate}
\usepackage{dsfont}

\usepackage[]{hyperref}
\hypersetup{
	bookmarksnumbered=true,
	colorlinks = true,
	linkcolor=blue,		
	anchorcolor=blue,	
	citecolor=blue,		
	filecolor=cyan, 		
	menucolor=blue,		
	urlcolor=blue,		
	pagebackref=true,	
	hyperfootnotes=false	
}

\newtheorem{lemma}{Lemma}
\newtheorem{theorem}{Theorem}
\newtheorem{remark}{Remark}

\newcommand{\R}{\mathbb{R}}
\newcommand{\N}{\mathbb{N}}
\newcommand{\E}{\mathbb{E}}
\newcommand{\PP}{\mathbb{P}}
\newcommand{\Var}{\operatorname{Var}}
\newcommand{\Cov}{\operatorname{Cov}}

\providecommand{\norm}[1]{\lVert #1 \rVert}

\begin{document}

\renewcommand{\thefootnote}{\fnsymbol{footnote}}

\title{Hurst index estimation in stochastic differential equations  driven by fractional Brownian motion}
\author{Jan Gairing\bigskip\\
	\it Mathematisches Institut\\
	\it Ludwig-Maximilians-Universit\"at M\"unchen\\
	\and 
	Peter Imkeller\bigskip\\
	\it Institut f\"ur Mathematik\\
	\it Humboldt-Universit\"at zu Berlin\\
	\and
	Radomyra Shevchenko\bigskip\\
	\it Fakult\"at f\"ur Mathematik\\
	\it Technische Universit\"at Dortmund\\
	\and
	Ciprian Tudor\bigskip\\
	\it D\'epartement de Math\'ematiques\\
	\it Universit\'e de Lille 1
	}

\maketitle

\begin{abstract}
We consider the problem of Hurst index estimation for solutions of  stochastic differential equations driven by an additive fractional Brownian motion. 
Using techniques of the Malliavin calculus, we analyze the asymptotic behavior of the quadratic variations of the solution, defined via higher order increments. Then we apply our results to construct and study estimators for the Hurst index.
\end{abstract}

\vskip0.3cm

{\bf 2010 AMS Classification Numbers:} 60G15, 60H05, 60G18.

\vskip0.3cm


\section{Introduction}
Fractional Brownian motion (fBm) is a widely used generalization of the standard Brownian motion that incorporates long-range dependence while preserving the self-similarity and Gaussianity.
These features make it a popular stochastic processes in the mathematical modelling of various complex systems from financial applications to surface growth (e.g. \cite{Bay, Bar}).
The need to detect and analyze such systems has led to the development of many, now classical, statistical instruments, including wavelet analysis, R/S estimators, generalized variations etc. and we refer to the monographs \cite{Ber,EM,T} for further reference.

Fractional Brownian motion  is homogeneous in space. Many physical systems, however, are subject to an external force and are confined to certain locations with high probability.
It may therefore be appropriate to consider stochastic differential equations (SDE) as a generalization of the fBm model. 
fBm then takes the role of a random forcing which heavily influences the underlying dynamics.
Thus the identification of this random forcing is of particular interest.

In our work, we will consider an SDE with additive fractional Brownian noise, i.e.
\begin{equation}
\label{eq:1}
X_{t}= x+\int_{0} ^{t} f(s, X_{s}) ds+ B ^{H} _{t}, \hskip0.5cm t\in T
\end{equation}
where $T$ is an interval in $[0, \infty)$, $x\in \mathbb{R}$ is the initial state and $f$ is a measurable deterministic function satisfying suitable assumptions. Stochastic models as \eqref{eq:1} appear in mathematical finance (for example, in \cite{KiTa} the log-volatility is assumed to satisfy an SDE of the form \eqref{eq:1}). 
Other potential applications come from climatology, stochastic processes as in \eqref{eq:1} (with Gaussian or non-Gaussian noise) are utilized as models for problems related to the earth's energy balance, see \cite{HIP,Dit}.

In fBm-related models particular interest lies in the estimation of the Hurst index (see Section~\ref{sec:prel} below) since it constitutes  the characteristic parameter of the driving forcing.

The purpose of our endeavour is to estimate the Hurst index $H$ in \eqref{eq:1} based on the discrete observations of the process $(X_{t}) _{t\in T}$.
Here we will deploy a well-known method, based on the quadratic variations of $X$.
This method is known to work well for self-similar processes (see e.g. \cite{T} and references therein).

We will show that it can be also applied to \eqref{eq:1}, although $X$ is not a self-similar process. 
The method exploits the fact that the absolutely continuous integral component of \eqref{eq:1} does not affect the roughness captured by the quadratic variation of $X$. 

We will define the sequence $V (a, n, \Delta, B ^{H})$ (see equation \eqref{eq:7} below) of the so-called {\it quadratic $a$-variations}, defined in terms of higher order increments of $X$ over {\it a filter} $a$. 
We give the limit behavior in distribution of this sequence making use of the Malliavin calculus to handle correlations and combine it with already known results concerning the variations of the fBm.
Interestingly, in this setting the case $H<\frac{1}{2}$ is easier to handle than the case of $H>\frac{1}{2}$. 
This is due to the fact that, when studying the limit of the sequence \eqref{eq:7}, one needs to take into account the correlations between the increments of the fBm, but also the correlations between the increments of the fBm and the increments of the Lebesque integral in \eqref{eq:1}. If $H<\frac{1}{2}$ these joint correlations are always dominated by those of the fBm, which is not the case for $H>\frac{1}{2}$, when a supplementary assumption is needed on the mesh $\Delta $ in \eqref{eq:7}.

By a standard procedure, we then construct a quadratic variation estimator for the Hurst index of the model \eqref{eq:1}. We prove its consistency and its asymptotic normality, using the limit behavior of the quadratic $a$-variations.

Our work is structured as follows. In Section~\ref{sec:prel} we introduce the objects of study, fBm, filters $a$ and the corresponding quadratic $a$-variations of a process. We also quote the underlying results on the behavior of the $a$-variations of fBm. 
Section~\ref{sec:sde} studies the properties and the $a$-variations of the solution to \eqref{eq:1}, via the techniques of the Malliavin calculus.
Section~\ref{sec:estim} is devoted to the parameter estimation of the Hurst index from discrete observations of $X$.
In Section~\ref{sec:simul} we present simulations and discuss statistical properties of the derived estimators.
In the Appendix we give a short review of the definitions and results from Malliavin calculus relevant to this work.

\section{Preliminaries: Fractional Brownian motion and its variations}\label{sec:prel}

A \emph{fractional Brownian motion} $(B^H_t)_{t\in T}$ with Hurst index $H\in (0,\,1)$ is a centered Gaussian process on the interval $T$ (in this work we consider $T=[0,\,1]$ or $T=\mathbb R^+$) with covariance function\
\begin{equation}\label{eq:2}
\mathbb E[B^H_t B^H_s]=\frac{1}{2}(t^{2H}+s^{2H}-|t-s|^{2H})
\quad \text{for}\ s,\,t\in T\ .\\
\end{equation}
From this definition it is apparent that for $H=\tfrac12$ we have that
$\mathbb E[B^H_t B^H_s]=\min(s,t)$ and the fBm with $H=\tfrac12$ is the standard Brownian motion. 

For a natural number $p\in\mathbb N$ let us consider a $(p+1)$-tuple $a=(a_0,\dots , a_p)$ of real numbers with zero sum, i.e.
\begin{equation*}
\sum_{q=0} ^{p} a_{q}=0\ .
\end{equation*}
Such a $(p+1)$-tuple $a=(a_0,\dots , a_p)$ will be called a {\it filter} of length $p+1$.
The \emph{order} of a filter $a$, denoted by $M(a)$, is defined to be the order of the first non-zero moment of the $(p+1)$-tuple $a$, i.e.\
\begin{equation}
\label{eq:3}
\sum_{i=0}^{p}a_i i^k=0 \text{ for }0\leq k < M(a)\quad\text{ and }\quad\sum_{i=0}^{p}a_i\, i^{M(a)}\neq 0\ .
\end{equation}
We will frequently need the partial sum of the components of $a$ and use the notation 
\begin{equation}\label{eq:4}
b_{i}=\sum_{k=0} ^{i} a_{k} \quad\mbox{for}\quad i=0,1,\dots,p.
\end{equation}
Since by definition a filter has zero sum, clearly for any filter $a$ we have $M(a)\geq 1$. For instance, $a=(a_{0}, a_{1}) = (-1, 1) $ is a filter of order 1 and of length 2, while $a=(a_{0}, a_{1}, a_{2})=(1,-2, 1) $ is a filter of order 2 and of length 3. 

For a process $(X_t)_{t\in T}$, a filter $a$ of length $p+1$ and a fixed \emph{mesh size} $\Delta>0$ we consider the sum\
\[\Delta_a X_j := \sum_{i=0}^p a_i X_{(i+j)\Delta}\]
 with $j\in\mathbb N$ such that $(p+j)\Delta \in T$. For example, if $a=(a_{0}, a_{1}) = (-1, 1) $, then $\Delta_a X_j =X_{(j+1)\Delta}- X_{j\Delta}$ while if  $a=(a_{0}, a_{1}, a_{2})=(1,-2, 1) $ then $\Delta_a X_j =\Delta_{j\Delta} -2\Delta _{(j+1) \Delta} + X_{(j+2)\Delta}$.

We will refer to $\Delta_a X_j$ as the \emph{increment of the process $X$ at time $j$ over the filter $a$}. 
Note that due to the zero sum condition on $a$ we can rewrite $\Delta_a X_j$ in the following way:\
\begin{equation}
\label{eq:5}
\begin{aligned}
\Delta_a X_j &= \sum_{i=0}^p a_i X_{(i+j)\Delta} \\
&=a_0 (X_{j\Delta}-X_{(j+1)\Delta})+(a_0+a_1)X_{(j+1)\Delta}+a_2 X_{(j+2)\Delta}+\dots+ a_p X_{(j+p)\Delta} =\cdots\\
&= a_0 (X_{j\Delta}-X_{(j+1)\Delta})+\dots +\sum_{i=0}^{p-1} a_i(X_{(j+p-1)\Delta}-X_{(j+p)\Delta})  +\underbrace{\sum_{i=0}^p a_i X_{(j+p)\Delta}}_{=0}\\
&=\sum_{i=0}^{p-1} \sum_{k=0}^i a_k (X_{(i+j)\Delta}-X_{(i+j+1)\Delta})
=\sum_{i=0}^{p-1} b_i (X_{(i+j)\Delta}-X_{(i+j+1)\Delta}).
\end{aligned}
\end{equation}
We will refer to this form as \emph{differences representation}.
The correlation of increments of the fBm over a filter plays an important role in our calculations. From \eqref{eq:2} we can see that for a zero sum vector $a\in\mathbb R^{(p+1)}$ and $i,\,j\in\N$ such that $\Delta_a B^H_i$, $\Delta_a B^H_j$ are well defined one has\
\begin{equation}
\label{eq:6}
\Cov(\Delta_a B^H_i,\,\Delta_a B^H_j)=-\frac{\Delta^{2H}}{2}\sum_{k=0}^p \sum_{l=0}^p a_k a_l |i+k-j-l|^{2H}.
\end{equation}

For a filter $a\in\mathbb R^{p+1}$ and a fractional Brownian motion $(B^H_t)_{t\in T}$ 
we define its {\it {(normalized)} quadratic $a$-variation} in the following way:\
\begin{equation}\label{eq:7}
V(a,\, n, \Delta, \,B^H)= \frac{1}{n}\sum_{j=1}^{n-p} \left(\frac{(\Delta_a B^H_j)^2}{\sigma_{a,\, \Delta}}-1\right) ,
\end{equation}
where $n\in\mathbb N$ with $n-p>0$ is the number of discrete observations 
on a grid of mesh size $\Delta >0$ (that may depend on $n$)
and 
\begin{equation}\label{dah}
\sigma_{a,\, \Delta}=\Var (\Delta_a B^H_j)=-\frac{\Delta^{2H}}{2}\sum_{k=0}^p \sum_{l=0}^p a_k a_l \vert k-l\vert ^{2H}.
\end{equation}
In what follows, we will assume that $\Delta=n^{-\alpha}$ for some $\alpha >0$ whenever the quadratic $a$-variation is considered. Observe that the choice $a=(-1,\,1)$ and $\alpha =1$ would yield the formula for the usual normalized quadratic variation of the fBm (see e.g. \cite{T}).

Let us recall the main result concerning the asymptotic behavior of the quadratic $a$-variation of fBm (see \cite{Coeur} and \cite{ILa}). It states that the sequence \eqref{eq:7} converges to zero almost surely for any filter $a$ and if $M(a)> H +\frac{1}{4}$ then  it satisfies a central limit theorem. 

\begin{theorem}\label{th:1}
Let $(B^H_t)_{t\in T}$ be a fractional Brownian motion with Hurst index $H\in (0, 1)$ and let $a$ be a filter of order $M(a)$ and of length $p+1$ with $p\geq 1$. Then
\begin{enumerate} 
\item $ V(a,\, n,\, \Delta,\, B^H) \stackrel{\text{a.s.}}{\to} 0 $ as $n\to \infty.$

\item If $M(a)> H+\frac{1}{4} $, then 
\[\sqrt{n} V(a,\, n,\, \Delta,\, B^H)\stackrel{(d)}{\to}N (0,\,\sigma_H),\]
where $\sigma_{H}>0$ is an explicit constant. 
\end{enumerate}
\end{theorem}
In the construction of the estimators in Section~\ref{sec:estim} we need a multidimensional version of this statement which can also be found in \cite{ILa} and \cite{Coeur}:

\begin{theorem}\label{th:2}
Let $(B^H_t)_{t\in T}$ be a fractional Brownian motion with Hurst index $H\in (0, 1)$.
For $i=1,\dots , N$ let $a^i$ be a filter of order $M(a^i)> H+\frac{1}{4} $.
Then
\[\sqrt{n} (V(a^1,\, n,\, \Delta,\, B^H),\dots, V(a^N,\, n,\, \Delta,\, B^H) ) \stackrel{(d)}{\to}N (0,\,\Sigma_H),\]
where $\Sigma_{H}$ is a positive definite matrix. 
\end{theorem}

Notice that, in the case of a filter of order 1, we have the well-known restriction $H<\frac{3}{4}$ for the CLT of the sequence $V(a,\, n,\, \Delta,\, B^H)$ to hold. 
This restriction can be lifted by taking a filter of order greater than one, i.e. by considering higher-order increments of the fBm. 

\section{Stochastic differential equations driven by the fractional Brownian motion}\label{sec:sde}

This section introduces the main object of our study: stochastic differential equations driven by the fractional Brownian motion. 
We will then study how the central limit theorems of the previous section carry over to the $a$-variation of the solution to the SDE. 

\subsection{Stochastic differential equations with fBm}

For $t\in T\subseteq \R^+$ we consider the SDE\
\begin{equation}\label{eq:8}
X_t = x + \int_0^t f(s, X_s) ds + B^H_t\ ,
\end{equation}
where $(B^H_t)_{t\in T}$ is a fractional Brownian motion with Hurst index $H\in(0,1)$ and $x$ is a real number. We assume that $f\in C^{0,\, 1}_b (T\times\R)$ (i.e., continuous in the time component and continuously differentiable in the space component as well as bounded, together with his partial derivative with respect to the second variable,  in both components) with\
\begin{equation}
\label{eq:9}
\norm{f}_{\infty}+\norm{f'}_{\infty}\leq M,
\end{equation}
for a fixed constant $M> 0$.  We denote by $f' $ the derivative of $f$ with respect to its second variable while $\Vert \cdot \Vert \infty$   stands for the infinity norm. \\

In the course of the paper we will be interested in two cases:\
\begin{itemize}
\item[(S1)]: the SDE \eqref{eq:8} with $T=[0,\,1]$,\
\item[(S2)]: the SDE \eqref{eq:8} with $T=\R^+$ for $H\leq \frac{1}{2}$.
\end{itemize}

The existence of pathwise solutions to \eqref{eq:8} in both cases can be seen by considering the process $X-B^H$ and thus reducing (S) to an ordinary differential equation. We refer to \cite{NO}, \cite{Bes} or \cite{Ngu}, among others, for the existence and uniqueness of the solution under assumption \eqref{eq:9}.

Clearly, solutions to (S2) also solve (S1), if restricted to the unit interval. Therefore statements made for (S2) also apply to (S1). 
We make this distinction mainly for notational reasons: in (S1) we consider only $\alpha \geq 1$ (where $\Delta = n^{-\alpha}$), such that we do not leave the unit interval as the number of observations ($n$) grows. 
In (S2) $\alpha$ is allowed to be less than one, and the observation window $n\Delta$ grows with $n$.
 
We further denote by $Y$ the absolutely continuous component of the equation \eqref{eq:8}, i.e.
\begin{equation}
\label{eq:10}
Y_t=\int_0^t f(s, X_s)ds, \mbox{ for every } t\in T.
\end{equation}

If $X$ is the solution to the SDE \eqref{eq:8} (in either the setting (S1) or (S2)) 
and $a\in\R^{p+1}$ is a filter of length $p+1$
we define the {\it quadratic $a$-variation} of $X$ to be
\[V(a,\, n, \Delta, \,X)= \frac{1}{n}\sum_{j=1}^{n-p} \left(\frac{(\Delta_a X_j)^2}{\sigma_{a,\, \Delta}}-1\right),\]
where the number of observations $n\in\N$ satisfies $n-p>0$ and the mesh size is chosen to satisfy $\Delta=n^{-\alpha}$ for some $\alpha >0$.
Here $ \sigma_{a,\, \Delta}=Var (\Delta_a B^H_j)$ is again the variance of the increment of $B^H$ along the filter  $a$.

\subsection{Central limit theorems for the quadratic $a$-variation of SDE}

This section comprises the core results of this work, combining the limit results of the fBm from Section~\ref{sec:prel} with estimates for the solution of SDE obtained via Malliavin calculus. 
We obtain central limit theorems for the SDE \eqref{eq:8} providing the basis for the estimation of $H$ to be discussed in Section~\ref{sec:estim}.
The results are presented in the Theorems~\ref{th:3} and \ref{th:4} where the cases $H<\frac{1}{2}$ and $H\in [\frac{1}{2},\,1)$ are treated separately.
The proof of Theorem~\ref{th:3} consists of direct estimates and does not require any assumptions on the order $M(a)$ of the filter $a$.

If $H\ge\frac{1}{2}$ these estimates are not sufficient and we rely on the Malliavin Calculus to provide a finer analysis of the correlation.
The actual proof of Theorem~\ref{th:4} is postponed to Section~\ref{sec:proof2} after some auxiliary results.


First we will consider the case $H<\frac{1}{2}$. Notice that in this case we have $M(a) >H+\frac{1}{4}$ for any filter $a$.\
\begin{theorem}\label{th:3}
Assume $H<\frac{1}{2}$, let $X$ be a solution of the SDE (S2) and let \textnormal{a} be a filter with $p+1$ components. For $\alpha > \frac{1}{2(1-H)}$ (with $\Delta = n^{-\alpha}$) we have\
\[\sqrt{n} V(a,\, n,\, \Delta,\, X)\stackrel{(d)}{\to} N (0,\, \sigma_H),\]
with $\sigma_H>0$ from Theorem~\ref{th:1}.
\end{theorem}
\begin{proof} 
By a binomial expansion we can write:\
\[\sqrt{n}V(a,\, n,\, \Delta,\, X) = \sqrt{n} V(a,\, n,\, \Delta,\, B^H) + R_n\]
with\
\[R_n=n^{-1\slash 2}\Delta ^ {-2H}\sum_{r=0}^1 C^{r} \sum_{i=0}^{n-p} (\Delta_a B^H_{i})^r (\Delta_a Y_{i})^{2-r}=:R_{n,\, 0}+R_{n,\, 1}\]
with some constants $C^0$, $C^1$.
Due to Slutsky's lemma and part (b) of Theorem~\ref{th:1} it is enough to show that $\E [|R_{n,\,r}|]\to 0$ for $r=0,\,1$. 
Using the differences representation \eqref{eq:5} we obtain with $b_{j}$ given by \eqref{eq:4}
\begin{equation}
\label{6s-1}
|\Delta_a Y_i|=\Bigl|\sum_{j=0}^{p-1}b_j\int_{(i+j)\Delta}^{(i+j+1)\Delta} f(s, X_s) ds \Bigr| \lesssim \Delta, 
\end{equation}
since $f$ is uniformly bounded by (\ref{eq:9}).  The  notation $f\lesssim g$ means that there exists a strictly positive constant $c$ such that $f\leq cg$.
Thus, we deduce for $R_{n,\,r}, r=0,1$:\
\[\E[|R_{n,\,r}|]\lesssim n^{-1\slash 2}\Delta^{-2H}\Delta ^{2-r}\sum_{i=0}^{n-p}\E[|\Delta_a B^H_i|^r].\]

Moreover, since for $i\in \{0,\dots n\}$, $\Delta_a B^H_i$ is normally\
 distributed, $\E[|\Delta_a B^H_i|]$ can be calculated directly:\
\[\E[|\Delta_a B^H_i|]=\frac{2\sqrt{\sigma_{a,\,\Delta}}}{\sqrt{2\pi}}\lesssim \Delta ^H,\]
since by \eqref{dah}, $\sigma_{a,\, \Delta}= -\Delta ^{2H}\underset{i>j}{\sum\limits_{i,\,j =0}^{p}} a_i a_j (i-j)^{2H}.$ Hence, we have\
 \[\E[|R_{n,\, 0}|]\lesssim \Delta^{-2H} n^{-1\slash 2}\Delta ^2 n= n^{1\slash2-\alpha (2-2H)} \]
 as well as\
 \[\E[|R_{n,\, 1}|]\lesssim \Delta^{-2H} n^{-1\slash 2}\Delta n \Delta^H= n^{1\slash2-\alpha (1-H)}, \]
both of which converge to zero because of our assumption on $\alpha$.\
\end{proof}

\begin{remark}
Note that for $H<1\slash 2$ one can choose $\alpha = 1$, thus considering the usual equidistant partition of the interval $[0,\, 1]$. But it is also possible to chose a mesh $\Delta$ less that $\frac{1}{n}$.\\
\end{remark}

In the second theorem the case $H\in\big[\frac{1}{2},\,1\big)$ is being considered. For this case we will assume $\alpha\geq 1$, which means that the observations do not leave the interval $[0,\,1]$. For $H\geq\frac{1}{2}$ the bounds on the correlation terms in the above proof fail to converge for $\alpha = 1$, therefore, we need more elaborate techniques to handle the remainder in Section~\ref{sec:proof2}. The assumption $\alpha\geq 1$ is purely technical: it permits the use of Lemma~\ref{lem:1} in Section~\ref{sec:aux}, the core estimate in our Malliavin calculus approach.

\begin{theorem}\label{th:4} 
Let $(B^H_t)_{t\in [0,\,1]}$ be a fractional Brownian motion with Hurst index $H\in\big[\frac{1}{2},\,1\big)$ and a filter $a$ satisfying $M(a) >H+\frac{1}{4}$. If the process $(X_t)_{t\in [0,\,1]}$ solves (S1)  and if $\alpha \geq 1$, $\alpha > \max \left(\frac{2H-1}{2-2H},\, \frac{1}{4-4H}\right)$
(with $\Delta = n^{-\alpha}$), we have\
\[\sqrt{n} V(a,\, n,\, \Delta,\, X) \stackrel{(d)} {\to} N (0,\, \sigma_H),\]
where $\sigma_H$ is the constant from Theorem~\ref{th:1}.\
\end{theorem}

\begin{remark}
All in all, for $H\in \left(\frac{1}{2},\,\frac{3}{4}\right)$ we have the convergence conditions $\alpha > \frac{1}{4-4H}$ and $\alpha > \frac{2H-1}{2-2H}$, both of which are satisfied for $\alpha \geq 1$. When $H>\frac{3}{4}$, we need to chose $\alpha >1$.
\end{remark}

\begin{remark}\label{rem:1}
The proofs of the Theorems \ref{th:3} and \ref{th:4} were based upon the demonstration that the differences $\sqrt{n}(V(a,\, n,\, \Delta,\, X)- V(a,\, n,\, \Delta,\, B^H))$ converge to zero in $L^1$. This implies that both theorems can be generalised to their multivariate versions, again by means of Slutsky's lemma and Theorem~\ref{th:2}.
\end{remark}

\subsection{Auxiliary results from Malliavin calculus}
\label{sec:aux}


We postpone to the Appendix the technical definitions and results from Malliavin calculus. 
We refer to  \cite{Bes} for the fact that the solution $X$ of (S1) belongs to the Sobolev-space $\mathbb{D}^{1,\,2}$, implying that the process $Y$ defined above is also Malliavin differentiable (compare also the proofs of Lemma~5.1 and Lemma~5.3 in \cite{Ngu} for a generalisation with respect to the time dependence of the drift). 
Moreover, the Malliavin derivative of $X$ (with respect to the fBm $B^ {H}$) satisfies $D_s X_t=0$ for $s>t$ and
\[D_s X_t = \int_s^t f'(r,\,X_r)D_s X_r dr + 1\]
for $s\leq t$. This linear equation has the explicit solution\
\[D_s X_t = e^{\int_s^t f'(r,\,X_r)dr},\]
and similarly to the equation considered in \cite{Bes} this provides the bound\
\[e^{-tM}\leq D_s X_t\leq e^{tM}\]
almost everywhere for $s\leq t$, where $M$ is the bound on $\|f'\|_{\infty}$ defined above in (\ref{eq:9}).\\

Moreover, we need the following technical result, which analyzes the correlation between the increments of the processes $B^{H}$ and $Y$. 
Below, $\mathcal{H}$ is the canonical Hilbert space of the fBm (see Section~\ref{app}).

\begin{lemma} 
\label{lem:1}
 Let $H \geq \frac{1}{2}$ and $Y$ be defined by \eqref{eq:10} for the SDE (S1).
 Then \
\begin{enumerate}
\item
for every $s,\,t,\,u,\,v\in [0,\,1]$ with $s<t$, $u<v$, we have
\[\E[|\left\langle D_\cdotp(Y_{u}-Y_{v}),\, 1_{[{s},\,{t}]}(\cdotp) \right\rangle_{\mathcal{H}}|]\lesssim ({t}-{s})({u}-{v})\ ,\]
\item 
 for every $s_i,\,t_i,\,u_i,\,v_i\in [0,\,1]$ with $s_i<t_i$, $u_i<v_i$, $i=1,\,2$.
\begin{align*}
\E[|&\left\langle D_\cdotp(Y_{u_1}-Y_{v_1}),\, 1_{[{s_1},\,{t_1}]}(\cdotp) \right\rangle_{\mathcal{H}}\left\langle D_\cdotp(Y_{u_2}-Y_{v_2}),\, 1_{[{s_2},\,{t_2}]}(\cdotp) \right\rangle_{\mathcal{H}}|]\\
&\lesssim (t_1-s_1)(u_1-v_1)(t_2-s_2)(u_2-v_2)
\end{align*}
\end{enumerate}
\end{lemma}

\begin{proof}
First note that for $s\in [0,\,1]$\
\[D_s(Y_u-Y_v)=D_s\int_v^u f(a, X_a) da =\int_v^u D_s f(a, X_a) da.\]
This can be seen by approximating the integrals using the fact that the integrands are bounded. 
Since the derivative operator is closed, it carries over to the limit.
The chain rule yields\
\[\int_v^u D_s f(a, X_a) da =\int_v^u \underbrace{f'(a, X_a)}_{|\dots|\leq M} \underbrace{D_s X_a}_{|\dots|\leq e^M} da,\]
which results in the bound\
\begin{equation}
\label{6s-2}
|D_s(Y_u-Y_v)|\lesssim (u-v).
\end{equation}
Now we will consider the case $H=\frac 1 2$, in which $B^H$ is the standard Brownian motion. 
Since the space $\mathcal{H}$ coincides with the space $ L^{2}(\Omega)$, we have 
\begin{eqnarray*}
&&\E [(|\langle D_\cdotp(Y_u-Y_v), \, 1_{[s,\,t]}(\cdotp)\rangle _{\mathcal{H}}|]=\E[|\langle D_\cdotp(Y_u-Y_v), \, 1_{[s,\,t]}(\cdotp)\rangle_{L^2}|]\\
&&\leq \E\left[\int_s^t |D_{\beta} (Y_u-Y_v)| d\beta \right]\lesssim (u-v)(t-s),
\end{eqnarray*}
applying in the last step the above bound (\ref{6s-2}) on $|D_{\beta} (Y_u-Y_v)|$.
Part (a) is proven for $H=\frac 1 2$. 
Part (b) is obtained similarly by the same arguments and bounds.

For the case $H> \frac 1 2$ we first calculate (recall that the norm in $\vert \mathcal{H}\vert $ is defined in (\ref{6s-3}))
\begin{eqnarray*}
\Vert D(Y_{u}-Y_{v}) \Vert _{\vert \mathcal{H} \vert } ^ {2} &=& \int_0^1 \int_0^1 |D_s(Y_u-Y_v)||D_r(Y_u-Y_v)||s-r|^{2H-2} ds dr\\
&\lesssim & (u-v)^2\underbrace{\int_0^1 \int_0^1 |s-r|^{2H-2} ds dr}_{=\frac{1}{\alpha_H}\langle 1_{[0,\,1]},\, 1_{[0,\,1]}\rangle_{\mathcal{H}}}<\infty,
\end{eqnarray*}
consequently, $D_s(Y_u-Y_v)\in |\mathcal{H}|$, which enables us to use the scalar product representation given in  \eqref{eq:17}. We get for the case (a), by (\ref{6s-2})\
\begin{eqnarray*}
\E[|\left\langle D_\cdotp(Y_u-Y_v),\, 1_{[s,\,t]}(\cdotp) \right\rangle_{\mathcal{H}}|]
&\leq \E\left[\alpha_H\int_0^1\int_0^1 |[D_{\beta}(Y_u-Y_v)| 1_{[s,\,t]}(\alpha) |\alpha-\beta|^{2H-2}d\beta d\alpha\right]\\
&=\alpha_H\int_s^t\int_0^1 \E[|D_{\beta}(Y_u-Y_v)|] |\alpha-\beta|^{2H-2}d\beta d\alpha \\
&\lesssim (u-v) \int_s^t\int_0^1 |\alpha-\beta|^{2H-2}d\beta d\alpha.
\end{eqnarray*}
Moreover, it holds that\
\begin{eqnarray*}
\int_s^t\int_0^1 |\alpha-\beta|^{2H-2}d\beta d\alpha & =& \frac{1}{\alpha_H}\langle 1_{[s,\,t]},\,1_{[0,\,1]}\rangle_{\mathcal{H}}\\
& = &\frac{1}{2}(t^{2H}-s^{2H})+\frac{1}{2} ((1-s)^{2H}-(1-t)^{2H})\lesssim (t-s),
\end{eqnarray*}
since for $0<a<b<1$ and $H>\frac{1}{2}$\
\[b^{2H}-a^{2H}=2H\int_0^{b-a}(x+a)^{2H-1}dx\leq 2H b^{2H-1} (b-a)\lesssim (b-a).\]
This proves part (a) of the lemma for $H>\frac 1 2$. The proof of part (b) follows analogously, since
\[\E[|D_{\beta_1}(Y_{u_1}-Y_{v_1})D_{\beta_2}(Y_{u_2}-Y_{v_2})|]\lesssim (u_1-v_1)(u_2-v_2).\]
\end{proof}

\subsection{Proof of Theorem~\ref{th:4}}\label{sec:proof2}

The correlation estimates of Lemma~\ref{lem:1} allow us to prove the theorem.

\begin{proof}
For both cases we consider the sum\
 $R_{n,\, 0}+R_{n,\,1},$
defined as in the proof of Theorem~\ref{th:3}. 
For the summand $R_{n,\,0}$ we get similarly to the previous theorem by the boundedness of $f$
\[\E[|R_{n,\,0}|] = Cn^{-1\slash 2}\Delta^{-2H} \sum_{i=0}^{n-p}\E[(\Delta_a Y_i)^2]\
\lesssim n^{-1\slash 2}\Delta^{-2H} n \Delta^2 = \Delta^{2-2H}n^{1\slash 2}=n^ {\alpha(2H-2)+\frac{1}{2}},\]
which converges to zero for $\alpha > \frac{1}{2(2-2H)}$.
We rewrite $R_{n,\,1}$ with the help of the differences representation \eqref{eq:5}.
Then
\begin{align}
R_{n,\,1}
&= Cn^{-1\slash 2}\Delta^{-2H}\sum_{j=0}^{n-p}(\Delta_a Y_j)(\Delta_a B^H_j)\nonumber \\
&=C n^{-1\slash 2}\Delta^{-2H}\sum_{j=0}^{n-p}(\Delta_a Y_j)(\sum_{i=1}^{p-1} b_i (B^H_{(i+j)\Delta}-B^H_{(i+j+1)\Delta}))\nonumber \\
&= -C n^{-1\slash 2}\Delta^{-2H}\sum_{j=0}^{n-p}\sum_{i=1}^{p-1}b_i(\Delta_a Y_j) \,  B^H(\bm 1_{[(i+j)\Delta,\,(i+j+1)\Delta]})\nonumber \\
&= C n^{-1\slash 2}\Delta^{-2H}\sum_{j=0}^{n-p}(\Delta_a Y_j) \,  B^H(h_j)\nonumber \\
&= C n^{-1\slash 2}\Delta^{-2H}\Big( \delta(u) + \sum_{j=0}^{n-p}    \langle D_\cdotp\Delta_a Y_j, \,  h_j \rangle_{\mathcal H}\Big)\  \label{6s-4}
\end{align}
ehere $\delta$ denotes the Skorohod integral.  This decomposition follows directly from the integration-by-parts formula \eqref{eq:18} applied to $u=\sum_{j=0}^{n-p} F_jh_j$ for
\[
F_j=\Delta_a Y_j\in\mathbb{D}^{1,\,2}, \qquad h_j=-\sum_{i=0}^{p-1}b_i 1_{[(i+j)\Delta,\, (i+j+1)\Delta]}(\cdotp), \quad j\in \{0,\dots ,n-p\}\ .
\]
Notice that  (see (\ref{eq:5})) $h_j = D_\cdotp \Delta_a B^H_j$.
We estimate the first summand in (\ref{6s-4}). We can write
\begin{align*}
& \E \left[\left|\left\langle  D_\cdotp \Delta_a Y_j,\,h_j\right\rangle _{\mathcal{H}}\right|\right]  \\
&= \E \left[\left|\left\langle D_\cdotp \left(\sum_{k=0}^{p-1} b_k (Y_{(k+j)\Delta}-Y_{(k+j+1)\Delta})\right),\,\sum_{i=0}^{p-1}\left(b_i 1_{[(i+j)\Delta,\,(i+j+1)\Delta]}(\cdotp)\right)\right\rangle _{\mathcal{H}}\right|\right]\\
& \leq \sum_{i,k =0}^{p-1}|b_i b_k| \cdot \E \left[\left|\left\langle  D_\cdotp (Y_{(k+j)\Delta}-Y_{(k+j+1)\Delta}),\,1_{[(i+j)\Delta,\,(i+j+1)\Delta]}(\cdotp)\right\rangle _{\mathcal{H}}\right|\right]\\
& \lesssim p^2\max_{0\leq k \leq p-1}\{ b_k^2\}\, \Delta^2 
\lesssim \Delta^2
\end{align*}
for any $j\in {0,\dots , n-p}$, where the last inequality follows from Lemma \ref{lem:1}, point 1. 
Hence we obtain
\begin{equation}
\E[|n^{-1\slash 2}\Delta^{-2H} \sum_{j=0}^{n-p}\langle D_\cdotp \Delta_a Y_j,\, h_j\rangle_{\mathcal{H}} | ]\lesssim  n^{1\slash 2}\Delta^{2-2H},
\label{eq:11}
\end{equation}
which goes to zero if $\alpha>\frac{1}{4-4H}$.
Let us turn to the first component $n^{-1\slash 2}\Delta^{-2H} \delta(u)$ and show that it vanishes in $L^2$ (implying convergence in probability).
\begin{equation}\label{eq:12}
\begin{aligned}
 &\E\bigl[ | \delta(u) |^2 \bigr] 
 =\E\bigl[ | \sum_{j=0}^{n-p} \delta((\Delta_aY_j)h_j) |^2 \bigr] 
 =\sum_{j,k=0}^{n-p} \E\bigl[ \delta((\Delta_aY_j)h_j) \cdot \delta((\Delta_aY_k)h_k) \bigr] \\
 &=\sum_{j,k=0}^{n-p} \Bigg( \E\bigl[ \bigl\langle (\Delta_aY_j)h_j,\, (\Delta_aY_k)h_k\bigr\rangle_{\mathcal H}\bigr] 
 + \E\bigl[ \bigl\langle D_\cdotp \Delta_aY_j,h_k \bigr\rangle_{\mathcal H}  \bigl\langle D_\cdotp\Delta_a Y_k,\, h_j\bigr\rangle_{\mathcal H}\bigr] \Bigg) \ ,
\end{aligned}
\end{equation}
which follows from a direct application of 
\eqref{eq:19} .
An expansion of the differences representation \eqref{eq:5} for $\Delta_aY_j, \Delta_aY_k$ and the definition of $h_j,h_k$ gives 
\begin{align*}
 &\E\bigl[ \bigl\langle D_\cdotp \Delta_aY_j,h_k \bigr\rangle_{\mathcal H}  \bigl\langle D_\cdotp\Delta_a Y_k,\, h_j\bigr\rangle_{\mathcal H}\bigr] \\
 &=\sum_{i,\,l,\,\mu,\,\nu=0}^{p-1}b_i b_l b_\mu b_\nu 
\cdot \E \Big[  \bigl\langle D_\cdotp (Y_{(i+j+1)\Delta}-Y_{(i+j)\Delta}),\, 1_{[(\nu+k)\Delta,\, (\nu+k+1)\Delta]}(\cdot)\bigr\rangle_{\mathcal{H}}\\
& \hspace{10em}\qquad \cdot \bigl\langle D_\cdotp (Y_{(l+k+1)\Delta}-Y_{(l+k)\Delta}),\, 1_{[(\mu+j)\Delta,\, (\mu+j+1)\Delta]}(\cdot)\bigr\rangle_{\mathcal{H}}\Big]\ .
\end{align*}
With the use of Lemma~\ref{lem:1} applied to each expectation this is easily dominated by
\begin{align*}
 \bigl| \E\bigl[ \bigl\langle D_\cdotp \Delta_aY_j,h_k \bigr\rangle_{\mathcal H}  \bigl\langle D_\cdotp\Delta_a Y_k,\, h_j\bigr\rangle_{\mathcal H}\bigr] \bigr|
 \lesssim p^4 \max_{0\leq k \leq p-1}\{ b_k^4\}\, \Delta^4 
 \lesssim \Delta^4\ .
 \end{align*}
Analogously we obtain
\begin{equation}
\label{eq:13}
\begin{aligned}
 &\bigl|\E\bigl[ \bigl\langle (\Delta_aY_j)h_j,\, (\Delta_aY_k)h_k\bigr\rangle_{\mathcal H}\bigr]\bigr| \\
&=\sum_{i,\,l,\,\mu,\,\nu=0}^{p-1} \bigl| b_i b_l b_\mu b_\nu 
\cdot \E \Big[  (Y_{(i+j+1)\Delta}-Y_{(i+j)\Delta})(Y_{(l+k+1)\Delta}-Y_{(l+k)\Delta}) \Big]\\
& \hspace{10em}\qquad \cdot \langle 1_{[(\mu+j)\Delta,\, (\mu+j+1)\Delta]}(\cdot),\, 1_{[(\nu+k)\Delta,\, (\nu+k+1)\Delta]}(\cdot)\rangle_{\mathcal{H}} \bigr| \\
&\le\max_{0\leq k \leq p-1}\{ b_k^4\} M^2\Delta^2\cdot \sum_{i,\,l,\,\mu,\,\nu=0}^{p-1}  
\bigl| \langle 1_{[(\mu+j)\Delta,\, (\mu+j+1)\Delta]}(\cdot),\, 1_{[(\nu+k)\Delta,\, (\nu+k+1)\Delta]}(\cdot)\rangle_{\mathcal{H}} \bigr| \ ,
\end{aligned}
\end{equation}
where we used the fact that the expectations are again bounded by $M^2\Delta^2$.

Observe that for fixed $i$, $l$, $\mu$, $\nu$ we will make the sum in \eqref{eq:12} larger if we estimate
\begin{equation}\label{eq:14}
\begin{aligned}
\sum_{j,k=0}^{n-p} & \E\bigl[ \bigl|\bigl\langle (\Delta_aY_j)h_j,\, (\Delta_aY_k)h_k\bigr\rangle_{\mathcal H}\bigr|\bigr] \\
&\lesssim \Delta^2\sum_{j,\,k=0}^{n-p+\eta} \E \bigl[ \bigl|\bigl\langle 1_{[j\Delta,\, (j+1)\Delta]}(\cdotp),\, 1_{[k\Delta,\, (k+1)\Delta]}(\cdotp)\bigr\rangle_{\mathcal{H}}\bigr|\bigr],
\end{aligned}
\end{equation}
where $\eta = \max (\mu,\,\nu)$.

Now the inner product satisfies
\[
\langle 1_{[i\Delta,\, (i+1)\Delta]}(\cdot),\, 1_{[j\Delta,\, (j+1)\Delta]}(\cdot)\rangle_{\mathcal{H}} 
=\Delta^{2H}\cdotp \frac{1}{2}\left(|i-j+1|^{2H}+|i-j-1|^{2H}-2|i-j|^{2H}\right)\ .
\]
We note that for $H=\frac{1}{2}$ the indicator functions are orthogonal. On the diagonal we have for $i=j$
\[
\langle 1_{[i\Delta,\, (i+1)\Delta]}(\cdot),\, 1_{[j\Delta,\, (j+1)\Delta]}(\cdot)\rangle_{\mathcal{H}} 
=\Delta^{2H}\ .
\]

For $H>\frac{1}{2}$ also the off-diagonal elements contribute. For $r:=i-j$ large we have $(r+1)^{2H}-2r^{2H}+(r-1)^{2H} \sim 2H(2H-1) r^{2H-2}$ since it uniformly approximates the second derivative of the function $x\mapsto x^{2H}$ at $x=r$.
Indeed, a Taylor expansion shows that for $r> 1$ 
\[
(r+1)^{2H}+(r-1)^{2H} - 2r^{2H} = 2H(2H-1)r^{2H-2} + O\left( (r-1)^{2H-2} \right)\ .
\]
We can now rearrange the off-diagonal part of the sum in \eqref{eq:14} and use this approximation:
\begin{align*}
\Delta^2&\sum_{\substack{j,\,k=0\\j\neq k}}^{n-p+\eta} \E \bigl[ \bigl|\bigl\langle 1_{[t_{j},\, t_{j+1}]}(\cdotp),\, 1_{[t_{k},\, t_{k+1}]}(\cdotp)\bigr\rangle_{\mathcal{H}}\bigr|\bigr]\\
=&\Delta^{2+2H} \frac{1}{2}2\sum_{r=1}^{n-p+\eta} \left(\left(n-p+\eta\right)+1-r\right)\left((r+1)^{2H}-2r^{2H}+(r-1)^{2H}\right)\\
\lesssim & \Delta^{2+2H}  \sum_{r=1}^{n-p+\eta} \left(\left(n-p+\eta\right)+1-r\right) r^{2H-2}.
\end{align*}
Recall now that for the sum $\sum_{k=1}^n \frac{1}{k^s}$ with $n\in\N$ and $s\in \R$ the following asymptotic equality stems from the Euler--Maclaurin formula. 
\[ \sum_{k=1}^n \frac{1}{k^s}\sim \zeta(s)-\frac{n^{1-s}}{s-1}\left(1+O\left(\frac{1}{n}\right)\right)\ , \] 
where $\zeta(s)$ is the Riemann zeta function. Hence,
\begin{align*}
\sum_{r=1}^{n-p+\eta} &\left(\left(n-p+\eta\right)+1-r\right) r^{2H-2}\\
& \lesssim n \sum_{r=1}^{n-p+\eta}  r^{2H-2}
 \lesssim n+n\cdotp n^{1-(2-2H)}\lesssim n^{2H}.
\end{align*}
As a final step we verify that correctly scaled we have for the second summand in (\ref{6s-4})
\begin{align*}
 &\E\bigl[ | n^{-1\slash 2}\Delta^{-2H} \delta(u) |^2 \bigr] \\
 &= n^{-1}\Delta^{-4H} \sum_{j,k=0}^{n-p} \Bigg( \E\bigl[ \bigl\langle (\Delta_aY_j)h_j,\, (\Delta_aY_k)h_k\bigr\rangle_{\mathcal H}\bigr] 
 + \E\bigl[ \bigl\langle D_\cdotp \Delta_aY_j,h_k \bigr\rangle_{\mathcal H}  \bigl\langle D_\cdotp\Delta_a Y_k,\, h_j\bigr\rangle_{\mathcal H}\bigr] \Bigg) \\
 &\lesssim  n^{-1}\Delta^{-4H} \Bigg(  \Delta^{2H+2}n^{2H} +  n^2 \Delta^4 \Bigg)
 = n \Delta^{4-4H}+ n^{2H-1}\Delta^{2-2H} \ .
\end{align*}
Substituting $\Delta=n^{-\alpha}$ gives us the condition $\alpha>\frac{1}{4-4H}$ as above for the vanishing of the first term and the aditional condition
$\alpha>\frac{2H-1}{2-2H}$ for the second. This bound combined with (\ref{eq:11}) and (\ref{6s-4}) will imply the conclusion.
\end{proof}


\section{Estimation of the Hurst parameter}\label{sec:estim}

Here we will construct estimators for the Hurst parameter $H$ of the Brownian motion $B^H$ driving the SDE (S2) and derive their properties, using the previous results. In particular, we will transfer the almost sure convergence result from Theorem~\ref{th:1} to the quadratic $a$-variation of the solution of the SDE (S2) and apply the central limit theorems from the previous section. 

We shall fix in the sequel a filter  $a=(a_0,\dots , a_p)$ is a filter of order $M(a)$. \\

For a stochastic process $(Z_t)_{t\in [0,\,1]}$ we define
\[U(a,\,n,\,\Delta ,Z):=\frac{1}{n}\sum_{j=1}^{n-p} (\Delta_a Z_j)^2\]
with $n$ such that $n-p >0$ and $\Delta = n^{-\alpha}$ for some $\alpha > 0$.
Then
\[V(a,\,n,\,\Delta ,B^H)=\frac{U(a,\,n,\,\Delta ,B^H)}{\sigma_{a,\, \Delta}}-\frac{n-p}{n}\]
holds, and we deduce from Theorem~\ref{th:1} that\
\[\frac{U(a,\,n,\,\Delta ,B^H)}{\sigma_{a,\, \Delta}} -1\stackrel{\text{a.s.}}{\to}0.\]
The following theorem shows the same result for the solution $X$ of (S1).\\

\begin{theorem}\label{th:5}
Let $(B^H_t)_{t\in [0,\,1]}$ be a fractional Brownian motion with Hurst parameter $H\in (0,1)$ and let $(X_t)_{t\in [0,\,1]}$ be the solution of the SDE (S2) driven by the fBm  $B^H$. Then\
\[\frac{U(a,\,n,\,\Delta ,X)}{\sigma_{a,\, \Delta}} -1\stackrel{\text{a.s.}}{\to}0.\]
\end{theorem}
\begin{proof}
We have\
\[U(a,\,n,\,\Delta ,X)=U(a,\,n,\,\Delta ,B^H)+U(a,\,n,\,\Delta ,Y)+2\frac{1}{n}\sum_{j=1}^{n-p} (\Delta_a Y_j)(\Delta_a B^H_j),\]
so it is enough to show that the last two summands divided by $\sigma_{a,\, \Delta}$ ($\sim \Delta^{2H}$) converge to zero almost surely. Recall that we have from (\ref{6s-1})  $|\Delta_a Y_j|\lesssim \Delta$. It follows that\
\[\left|\frac{U(a,\,n,\,\Delta ,Y)}{\sigma_{a,\, \Delta}}\right|\stackrel{\text{a.s.}}{\lesssim}\Delta^{-2H} n^{-1}n\Delta^2=\Delta^{2-2H},\]
which goes to $0$ for every $\alpha$ as $n$ tends to infinity.\\

To estimate $\Delta_a B^H_j$ we refer to \cite{Nua} for the fact that almost all sample paths of a fractional Brownian motion with the Hurst index $H$ are H\"older continuous of order $H-\varepsilon$ for any $\varepsilon > 0$ (it follows from the Kolmogorov's continuity theorem, since $B^H$ is self-similar). We get for almost every trajectory, using the differences representation \eqref{eq:5}:\
\[|\Delta_a B^H_j|\leq \sum_{i=0}^{p-1}|b_i||B_{(i+j+1)\Delta}-B_{(i+j)\Delta}|\leq \sum_{i=0}^{p-1}|b_i| C \Delta^{H-\varepsilon}\lesssim \Delta^{H-\varepsilon}.\]
Note that the constant $C$ can vary depending on the trajectory. Then we get\
\[\left|\frac{1}{n\sigma_{a,\, \Delta}}\sum_{j=1}^{n-p} (\Delta_a Y_j)(\Delta_a B^H_j)\right|\stackrel{\text{a.s.}}{\lesssim} n^{-1}\Delta^{-2H}n\Delta \Delta^{H-\varepsilon}=\Delta^{1-H-\varepsilon}.\]
We can ensure that it converges to zero by setting $\varepsilon := \min \left(\frac{1-H}{2},\,\frac{H}{2}\right)$, and the claim follows.\
\end{proof}

Note that for the special case $a=(-1,\,1)$ and $\Delta = n^{-1}$ we can deduce directly from this result that the ''standard'' estimator $\hat{H}$ given by
\[\hat{H}:=\frac{\log (n U(a,\,n,\,\Delta,\,X))}{-2\log (n)}+\frac{1}{2}\]
converges almost surely to $H$.

\begin{remark}\label{rem:2}
Since the relationship between $V(a,\,n,\,\Delta ,X)$ and $U(a,\,n,\,\Delta ,X)$ is the same as for the process $B^H$, it follows from Theorem~\ref{th:5} that $V(a,\,n,\,\Delta ,X)\stackrel{\text{a.s.}}{\to}0$.\
\end{remark}
Now we are in a position to construct an estimator $\hat{H}_1$ in the way it was done in \cite{ILa} (see also \cite{Coeur}) for processes with stationary increments. We know from Theorem~\ref{th:5} that $U(a,\,n,\,\Delta ,X)$ asymptotically equals $\sigma_{a,\,\Delta}$, and a simple rearrangement argument gives us\
\[\sigma_{a,\,\Delta}=-\frac{1}{2}\Delta^{2H}\sum_{k,\,l =0}^{p}a_k a_l |k-l|^{2H}=-\sum_{d=1}^{p} (\Delta d)^{2H}\sum_{k=0}^{p-d}a_k a_{k+d}.\]

Let us consider for $i\in \{1,\dots ,m\}$ a family of filters $a^{(i)}=(a_0^{(i)},\dots , a_{p_i}^{(i)})$.
For each member the associated variance $\sigma_{a,\Delta}$ is a linear combination of the $ (\Delta d)^{2H}$, $d\in \{1,\dots ,P\}$, where $P=\max_i p_i$, the numbers $p_i$ being not necessarily different. So if we consider a matrix $A=(A_{ij})\in \R^{m\times P}$ defined by\
\[A_{ij}=-\sum_{k=0}^{p_j-j}a^{(i)}_k a^{(i)}_{k+j}\text{ for }j\leq p_j\text{ and }A_{ij}=0\text{ otherwise},\]
then $U_n:=(U(a^{(i)},\,n,\,\Delta ,X))_{i=1,\dots ,m}$ tends almost surely to $AD$, where $D$ denotes the vector $((\Delta d)^{2H})_{d=1,\dots ,P}$.\\ 
We choose the $a^{(i)}$ such that $A$ is of full rank and estimate $D$ by linear regression (preserving the almost sure convergence property): $\hat{D}=(A^TA)^{-1} A^T U_n$. Now we can consider $\log |\hat{D}|$, which tends to $2H (\log (d \Delta))_{d=1,\dots ,P}$, and estimate $H$ by another simple linear regression:\
\[\hat{H}_1=\frac{\sum\limits_{d=1}^P \log |\hat{D}_d|\log(d\Delta)}{2\sum\limits_{d=1}^P \log(d\Delta)^2 }.\]
If we know the observations to be equidistant but do not have access to the actual mesh size $\Delta$, we can estimate $D$ in the same way (because $U_n$ does not directly depend on $\Delta$) and then perform a regression with the intercept term $2H \log (\Delta)$. This allows the estimation\
\[\hat{H}_2=\frac{\sum\limits_{d=1}^P \log |\hat{D}_d|\log(d)-\frac{1}{P}\sum\limits_{d=1}^P \log |\hat{D}_d|\sum\limits_{d=1}^P \log(d)}{2 \left(\sum\limits_{d=1}^P \log(d)^2-\frac{1}{P}\left(\sum\limits_{d=1}^P \log(d)\right)^2\right)}.\]
To establish some properties of the estimators $\hat{H}_1$ and $\hat{H}_2$, we first recall that, as a consequence of Theorem~\ref{th:5}, we have the convergence result $U_n\stackrel{\text{a.s.}}{\to}AD$ and therefore\
\[\hat{D}=(A^TA)^{-1} A^T U_n\stackrel{\text{a.s.}}{\to} D.\]
Then, evoking the continuous mapping theorem, we obtain $\hat{H}_1\stackrel{\text{a.s.}}{\to}H$, $\hat{H}_2\stackrel{\text{a.s.}}{\to} H$. 
In other words, both estimators are strongly consistent.

The question of asymptotic normality depends upon whether Theorems \ref{th:3} and \ref{th:4} are applicable. 
Choosing filters of order 2 one can ensure that the filter condition is always satisfied, but other assumptions require more care. In particular, if there is no prior knowledge about the true value of $H$, one has to assume that the mesh size condition $\alpha > \frac{1}{4-4H}$ is satisfied. 
If a certain bound for $H$ is known, say, if $H\in \left(0,\,\frac{3}{4}\right)$, then under the assumption that $\alpha \geq 1$ (which includes the usual partition of $[0,\,1]$) we obtain asymptotic normality. The following theorem summarizes our observations.

\begin{theorem}
Let $X$ be a solution of the SDE (S2) and let the mesh size $\Delta$ be chosen such that the assumptions of Theorems \ref{th:3} and \ref{th:4} are satisfied. Moreover, let $a^{(i)}$ be filters of respective lengths $p^{(i)}+1$ (for $i=1,\dots , m$) such that $M(a^{(i)})> H+ \frac{1}{4}$. Then the sequences $\sqrt{n} (\hat{H}_1-H)$ and $\sqrt{n} (\hat{H}_2-H)$ converge weakly to normally distributed, centered random variables.
\end{theorem}
\begin{proof}
This is an application of the delta method (mentioned, for example, in \cite{DCD}) for the vector $V_n:=(V(a^{(i)},\,n,\,\Delta ,X))_{i=1,\dots ,m}$. Its almost sure convergence was shown in Remark \ref{rem:2} and the multivariate convergence in law of $\sqrt{n}V_n$ follows due to Remark \ref{rem:1}. Since for the construction of $\hat{H}_1$, $\hat{H}_2$ it underwent only linear and logarithmic transformations, the obtained estimators are indeed asymptotically normal.
\end{proof}
%
%

\section{Simulation study}\label{sec:simul}
Similarly to \cite{ILa} and \cite{Coeur} two filters $(a^{(1)}_i)_{i\in \{0,\dots ,p \}}$ and $(a^{(2)}_i)_{i\in \{0,\dots ,2p \}}$ are considered, where $a^{(2)}$ is obtained by ''thinning'' the filter $a^{(1)}$ (i.e., $a^{(2)}_{2k}:= a^{(1)}_k$ for $k\in \{0,\dots ,p \} $ and zero otherwise). In this case the estimator $\hat{H}_1$ simplifies to\
\[\hat{H}_1 =\frac{1}{2}\log_2 \left(\frac{U(a^{(2)}, n, \Delta , X)}{U(a^{(1)}, n, \Delta , X)} \right). \]
Note that this estimator is independent of the scaling: its strong consistency has been shown for all $\alpha >0$, hence, it is enough to know that the observations are equidistant, and the mesh size $\Delta$ can then be chosen appropriately. It is, moreover, by construction independent of deterministic multiplicative scaling factors of the fBm involved, which allows us to include the (slightly more general) case
\begin{equation}\label{eq:15}
X_t = x + \int_0^t f(s, X_s) ds + \sigma B^H_t
\end{equation}
with some unknown $\sigma >0$ and observed over an unknown interval in our simulation study.
From these observations different settings for the simulations arise.

\begin{figure}[h]
\label{fig:paths}
\centering
\includegraphics[width=0.3\textwidth]{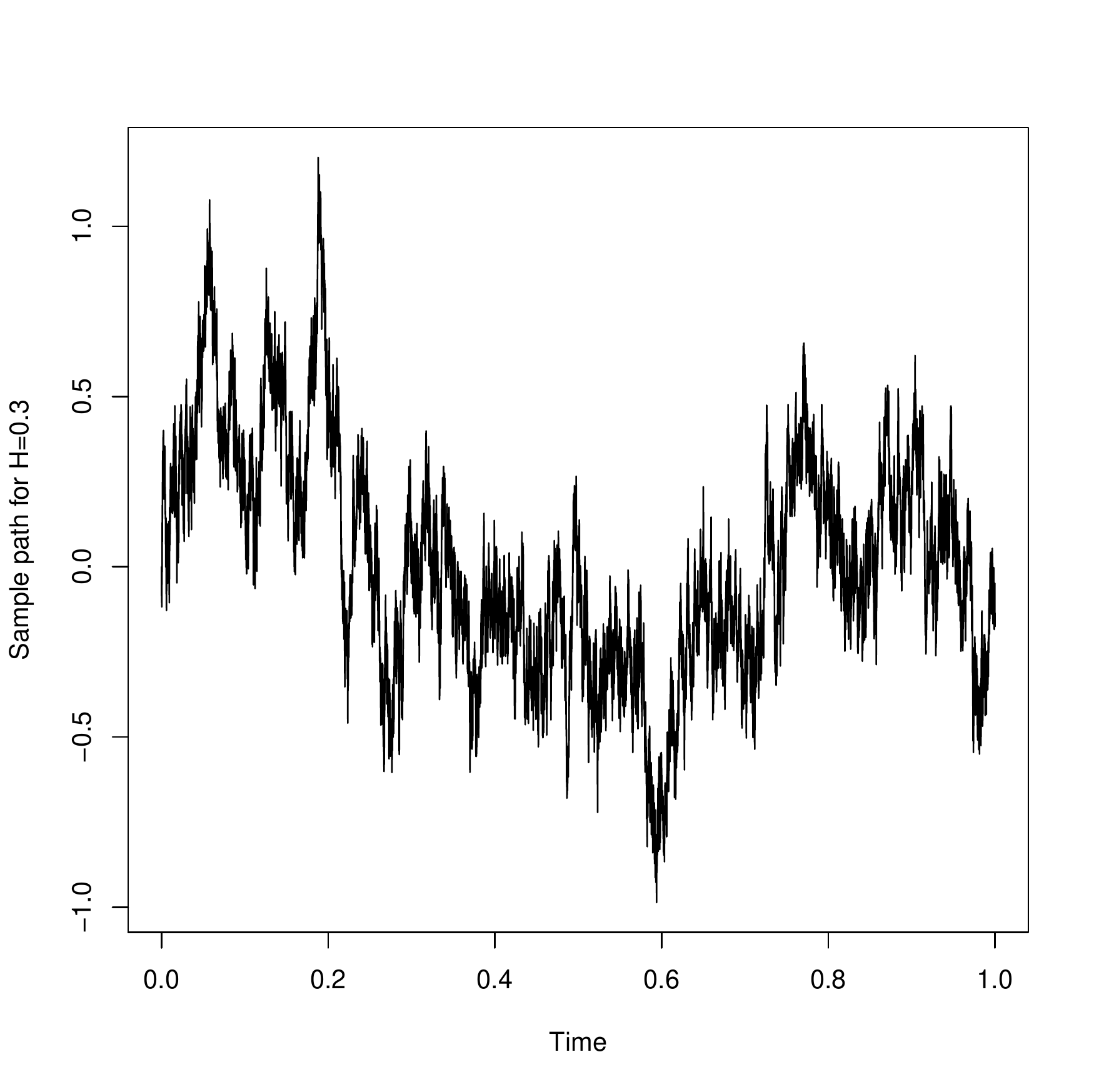}
\includegraphics[width=0.3\textwidth]{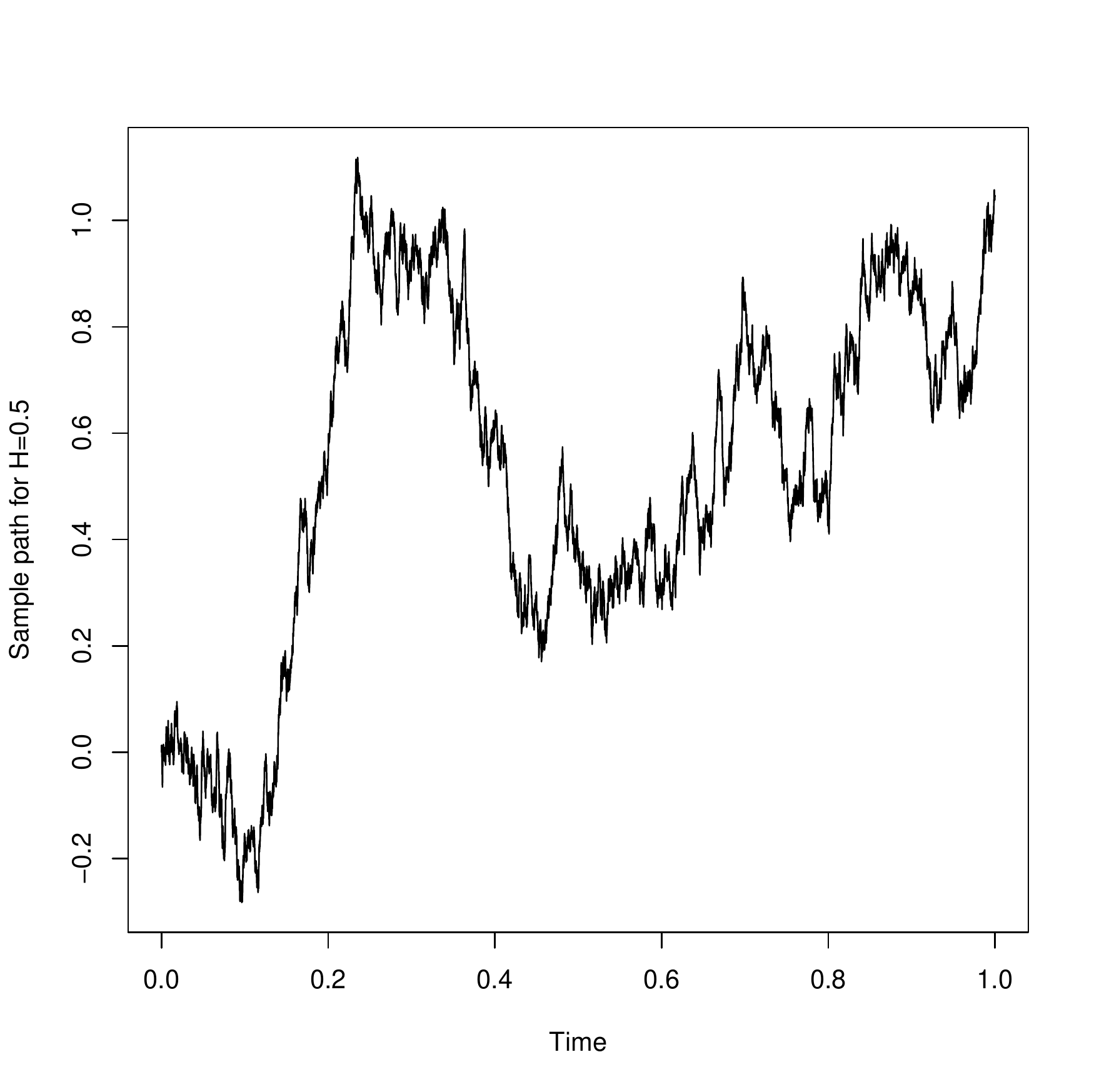}
\includegraphics[width=0.3\textwidth]{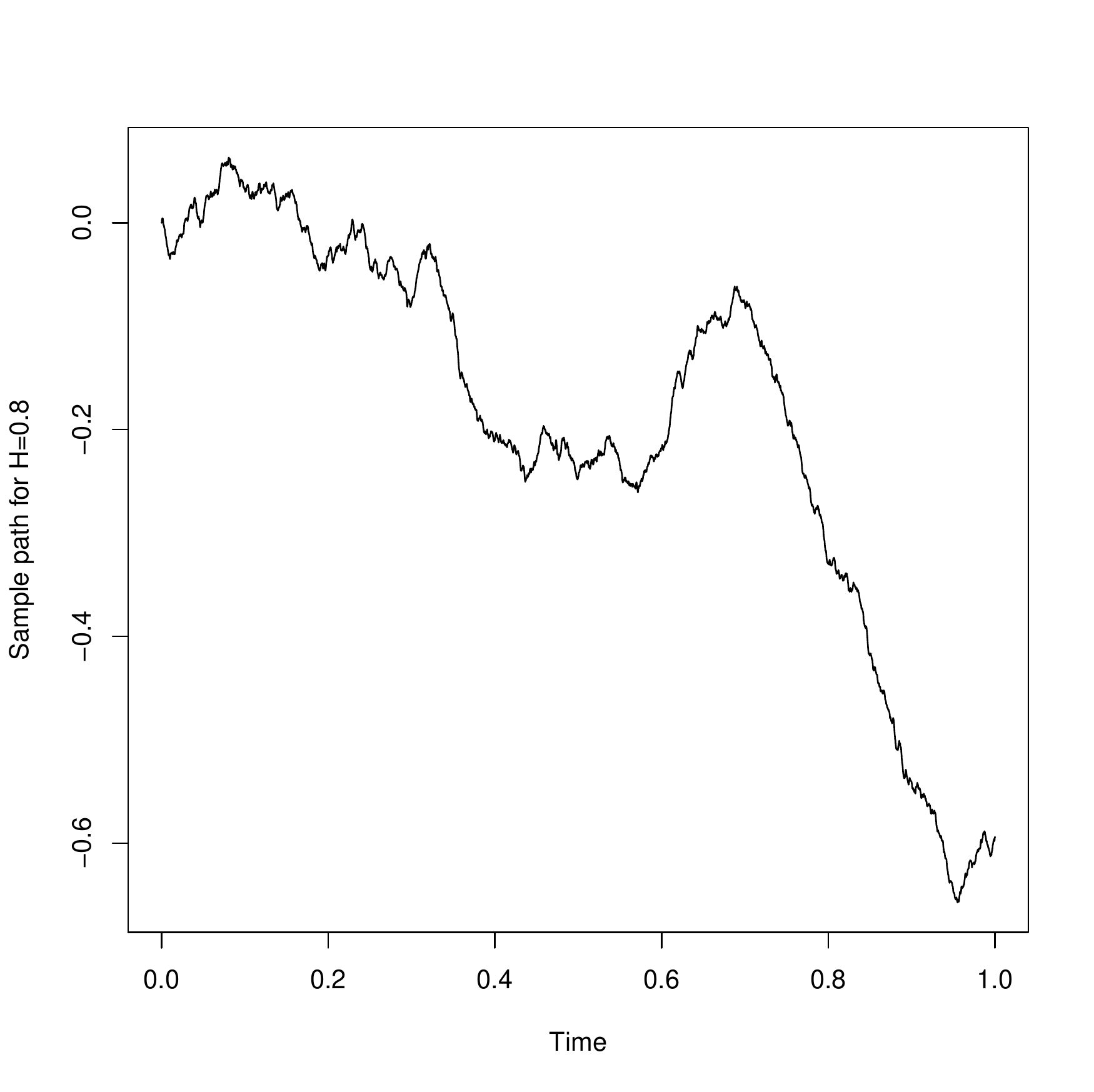}
\caption{Simulated paths of the SDE \eqref{eq:16} for $H=0.3$, $0.5$ and $0.8$ respectively.}
\end{figure}

We simulate $100$ trajectories of a process $X$ defined by
\begin{equation}
dX_t = \sin (X_t + t)dt+\sigma dB^H_t,\quad X_0=0.
\label{eq:16}
\end{equation}
Some trajectories are displayed in Figure~\ref{fig:paths} for different values of $H$.
We calculate the MSE for $1000$, $2000$, $4000$ and $8000$ observations in the following settings:
\begin{enumerate}
\item[(S1)] $H=0.7$ on an interval $[0,\,1]$ with the ''standard'' estimator $\hat{H}$, $\sigma = 1$,
\item[(S2)] $H=0.7$ on an interval $[0,\,1]$ with $\hat{H}_1$ for $a^{(1)}:=(-1,\,1)$, $\sigma = 1$,
\item[(S3)] $H=0.7$ on an interval $[0,\,1]$ with $\hat{H}_1$ for $a^{(1)}:=(-1,\,1)$, $\sigma = 5$,
\item[(S4)] $H=0.7$ on an interval $[0,\,10]$ with $\hat{H}_1$ for $a^{(1)}:=(-1,\,1)$, $\sigma = 1$,
\item[(S5)] $H=0.98$ on an interval $[0,\,1]$ with the ''standard'' estimator $\hat{H}$, $\sigma = 1$,
\item[(S6)] $H=0.98$ on an interval $[0,\,1]$ with $\hat{H}_1$ for $a^{(1)}:=(1,\,-2,\,1)$, $\sigma = 1$,
\item[(S7)] $H=0.98$ on an interval $[0,\,10]$ with $\hat{H}_1$ for $a^{(1)}:=(1,\,-2,\,1)$, $\sigma = 1$.
\end{enumerate}
We obtain the following results:
\begin{center}
    \begin{tabular}{ c | c | c | c | c | c | c | c |}
    n & (S1) & (S2) & (S3) & (S4) & (S5) & (S6) & (S7) \\ \hline
    $1000$ & $2.23\cdotp 10^{-5}$ & $0.00051$ & $0.0004$ & $0.0005$ & $0.004$ & $0.001$ & $0.00085$ \\ \hline
    $2000$ & $9.27\cdotp 10^{-6}$ & $0.00027$ & $0.0002$ & $0.00025$ & $0.003$ & $0.00074$ & $0.00054$ \\ \hline
    $4000$ & $4.31\cdotp 10^{-6}$ & $0.00014$ & $0.00011$ & $0.00012$ & $0.003$ & $0.00034$ & $0.00025$ \\
    \hline
    $8000$ & $1.82\cdotp 10^{-6}$ & $7.67\cdotp 10^{-5}$ & $4.8\cdotp 10^{-5}$ & $5.95\cdotp 10^{-5}$ & $0.0025$ & $0.00014$ & $0.00015$ \\
     \hline   
    \end{tabular} .
\end{center}
Indeed, one can see that the estimators perform well for a broad range of conditions. For $H<\frac{3}{4}$, in case of known $\sigma$ and on the unit interval the simplest estimator $\hat{H}$ seems to perform best. However, for large $H$ its convergence becomes slower compared to that of $\hat{H}_1$ for the order $2$ filter. It is a known result (see e.g. \cite{BM}) that for the filter $(-1,\,1)$ the asymptotic rates of convergence for the quadratic variation of an fBm are $\sqrt{n\log (n)}$ in the case $H=\frac{3}{4}$ and $n^{2H-1}$ for $H>\frac{3}{4}$. This constraint seems to transfer to the process $X$. Figures 1 and 2 provide a visualisation for the difference in the convergence rates.

\begin{figure}[h]
\centering
\includegraphics[width=0.45\textwidth]{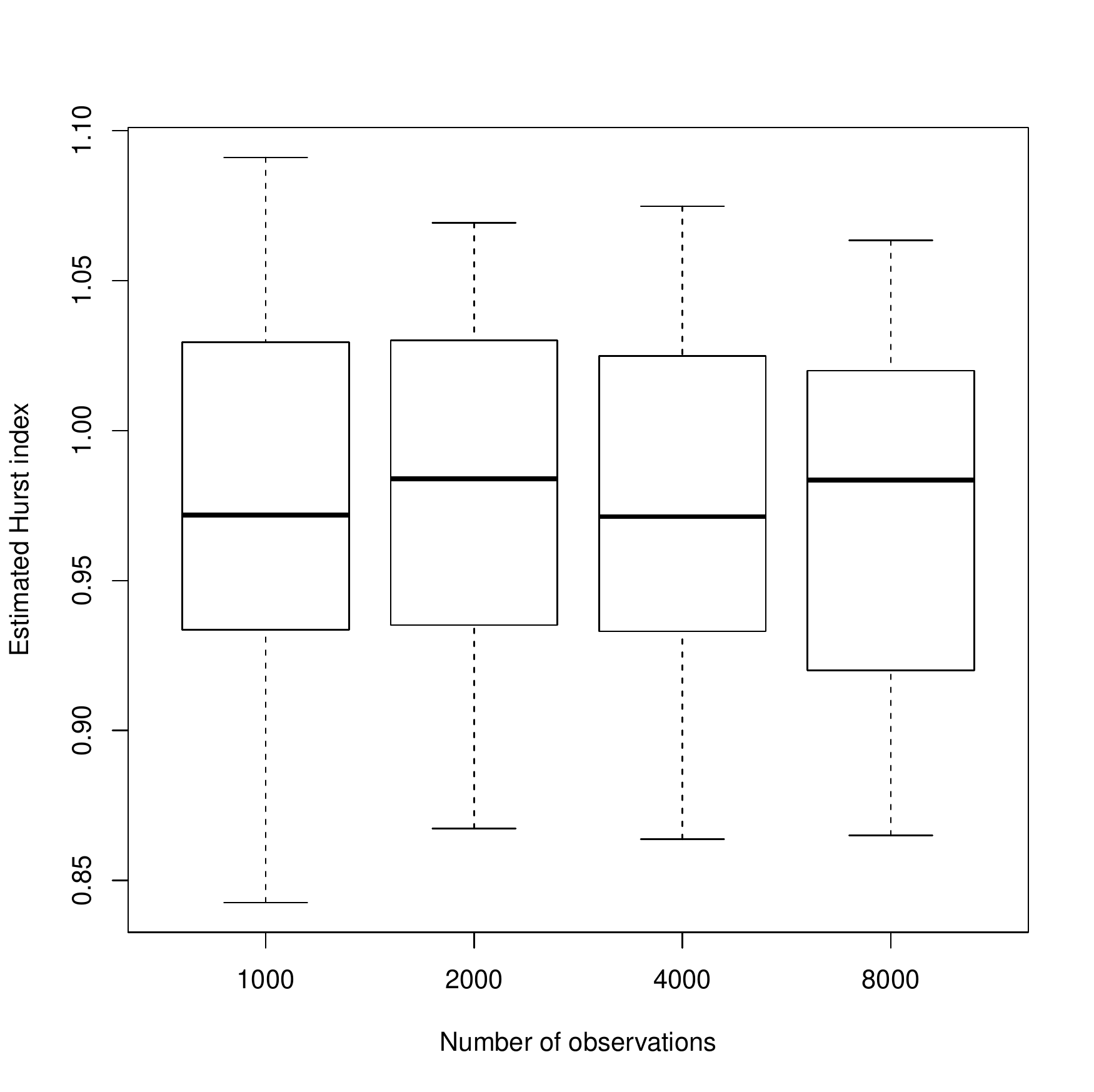}
\includegraphics[width=0.45\textwidth]{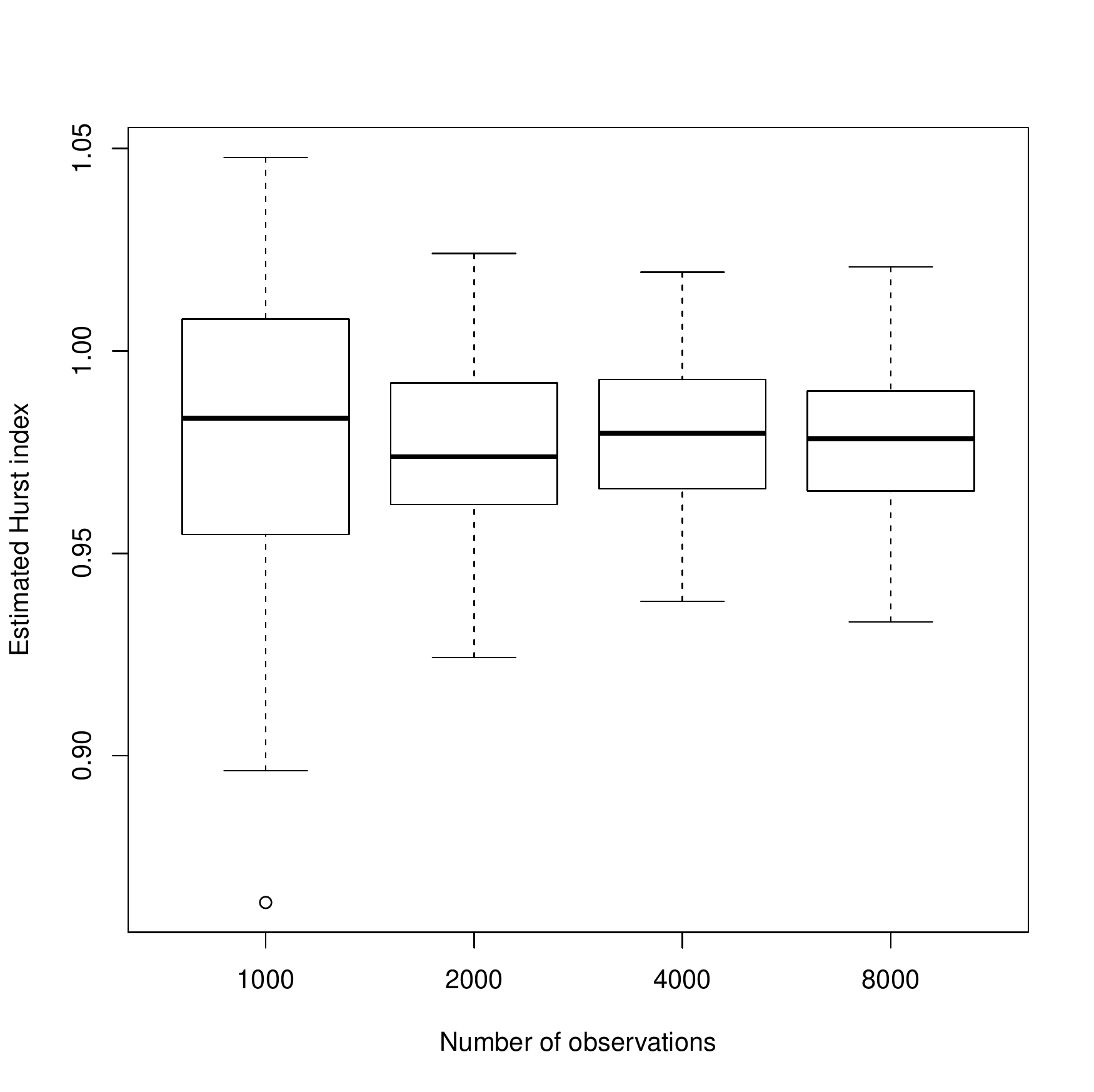}
\caption{A boxplot diagram for $\hat{H}$ in the setting (S5) and (S6).}
\end{figure}

\section{Appendix: Definitions and standard results from Malliavin calculus}\label{app}
In the following we will talk about the main definitions of Malliavin calculus for the fractional Brownian motion case, following mainly \cite{Nua}. Throughout this section we will consider $T=[0,\,1]$.\\

For a real separable Hilbert space $(\mathcal{H},\,\langle \cdotp ,\, \cdotp\rangle_{\mathcal{H}})$ we call a stochastic process $W=\{W(h),\,h\in \mathcal{H}\}$ in a complete probability space $(\Omega,\,\mathcal{F},\,\PP)$ a Gaussian process on $\mathcal{H}$ if $W$ is a centred Gaussian family of random variables such that for all $f,\,g\in \mathcal{H}$
\[\E[W(f)W(g)]=\langle f,\,g\rangle_{\mathcal{H} }.\]
For a fractional Brownian motion $(B^H_t)_{t\in T}$ with Hurst parameter $H\in (0,\,1)$ let $\mathcal{H}$ be the closure of the set of indicator functions with respect to the inner product\
\[ \langle 1_{[0,\,s]},\,1_{[0,\,t]}\rangle_{\mathcal{H}}:= \frac{1}{2}(t^{2H}+s^{2H}-|t-s|^{2H}).\]
Then $B^H$ is via $B^H(1_{[0,\,s]})=B^H_s$ by definition a Gaussian process on $\mathcal{H}$, and this is how the notation $\mathcal{H}$ is used in this work.\\

For $H>\frac{1}{2}$ the space $\mathcal{H} $ does not only contain functions (see \cite{PTa1}), so we consider the following definition:\
\begin{equation}
\label{6s-3}|\mathcal{H} |:=\left\{f: [0,\,1]\to \R \text{ meas. s.t. }\Vert f\Vert _{\vert \mathcal{H}\vert } ^ {2}:= \int_0^1 \int_0^1 |f(v)||f(u)||u-v|^{2H-2}dvdu<\infty\right\}.
\end{equation}

It has been shown in \cite{PTa2} that $|\mathcal{H}|$ is a subspace of $\mathcal{H}$, which yields a sufficient condition for functions to belong to the space $\mathcal{H} $.\\

We have a useful representation of the inner product of two functions from the space $|\mathcal{H}|$. That is, if $H>\frac{1}{2}$, then for $f,\,g\in |\mathcal{H}|$ (see e.g. \cite{Nua})
\begin{equation}
\label{eq:17}
\langle f,\,g\rangle_{\mathcal{H}}= \underbrace{H(2H-1)}_{=:\alpha_H} \int_0^1 \int_0^1 f(v)g(u)|u-v|^{2H-2}du dv.
\end{equation}

For $H=\frac{1}{2}$ (that is, if $B^H$ is the usual Brownian motion) the space $\mathcal{H}$ is identical with the space $L^2([0,\,1])$ endowed with the usual inner product.\

For the Gaussian process $(B^H_t)_{t\in T}$ on $\mathcal{H}$ let $\mathcal{S}$ denote the set of smooth random variables of the form\
\[\mathcal{S}:=\left\{g (B^H(h_1),\dots ,B^H(h_n)) | g\in C^\infty _p (\R^n),\, h_1,\dots ,h_n\in \mathcal{H},\, n\geq 1 \right\},\]
where $C^\infty _p (\R^n)$ denotes the space of infinitely continuously differentiable functions $g:\R^n\mapsto \R$ such that $g$ and all its partial derivatives grow at most polynomially.\\
Then we can define the Malliavin derivative $D$ as an operator which acts on functions $F=g (B^H(h_1),\dots ,B^H(h_n)) \in \mathcal{S}$ by\
\[DF=\sum_{i=1}^n \frac{\partial g}{\partial x_i}(B^H(h_1),\dots , B^H(h_n))h_i,\]
which means that $DF$ is a random variable with values in $\mathcal{H}$. It does not depend on the choice of the function $g$ and of $h_1,\dots , h_n\in \mathcal{H}$.\\
For the subset of $\mathcal{S}$, whose functions $g$ have compact support, we will write $\mathcal{S}_0$.\\
$D$ is clearly linear and it holds: $D(B^H_t-B^H_s)=1_{[s,\,t]}(\cdotp)$.\\
Moreover, $D:\mathcal{S} \to L^2(\Omega;\,\mathcal{H})$ is closable (this has been shown in \cite{Nua}) and we will denote by $\mathbb{D}^{1,\, 2}$ the closure of $\mathcal{S}$ with respect to the norm\
\[ \|F\|_{1,\,2}:=\left(\E[|F|^2]+\E[\|DF\|^2_{\mathcal{H}}]\right)^{\frac{1}{2}}.\]
For a fixed $h\in \mathcal{H} $ we define the operator $D^h$ on the set $\mathcal{S}$ by\
\[D^h F=\langle DF,\, h\rangle_{\mathcal{H}}.\]
Since Malliavin derivatives are $\mathcal{H}$-valued random variables, one can identify them with stochastic processes if those values in $\mathcal{H}$ are functions. In this case we will write $D_s X$ for $DX (s)$.\\

We have a chain rule for Malliavin derivatives, which is stated in \cite{Nua}: Let $f\in C(\R^m; \R)$ be a function with $\|\partial_i f\|_{\infty}<M_i< \infty$ for some $M_i>0$ ($i=1,\dots ,m$) and let $X_1,\dots X_m\in\mathbb{D}^{1,\,2}$. Then $f(X_1,\dots , X_m)\in \mathbb{D}^{1,\,2}$ and\
\[Df(X_1,\dots , X_m)=\sum_{i=1}^m \partial_i f(X_1,\dots ,X_m)DX_i.\]

Let $\mathcal{S}_{\mathcal{H}}$ denote the set\
\[\mathcal{S}_{\mathcal{H}}=\left\{\sum_{i=1}^n F_i h_i\Big| F_1,\dots ,F_n\in\mathcal{S},\,h_1,\dots , h_n\in \mathcal{H},\,n\geq 1 \right\}.\]
Then for $u=\sum_{i=1}^n F_i h_i$ we can define $Du := \sum_{i=1}^n DF_i \otimes h_i$ and consider the norm\
\[\|u\|_{1,\,2,\,\mathcal{H}}=\left(\E[\|u\|^2_{\mathcal{H}}]+\E[\|Du\|^2_{\mathcal{H}\otimes \mathcal{H}}]\right)^{\frac{1}{2}}.\]
Now we can, just as for the space $\mathcal{S}$, consider the closure of $\mathcal{S}_{\mathcal{H}}$ with respect to this norm and call it $\mathbb{D}^{1,\,2}(\mathcal{H})$.\\

The divergence operator, denoted by $\delta$, is the adjoint of the derivative operator $D$. Its domain $\operatorname{Dom} (\delta )$ is the set of random variables $u\in L^2(\Omega;\, \mathcal{H})$ such that\
\[|\E[\langle DF,\, u\rangle_{\mathcal{H}}]|\leq c_u \|F\|_2\]
for all $F\in \mathbb{D}^{1,\,2}$, where $c_u$ is a constant.\\
For $u\in \operatorname{Dom} (\delta )$ $\delta (u)$ satisfies the following characterising equation\
\[\E[F\delta(u)]=\E[\langle DF,\, u\rangle_{\mathcal{H}}]\]
for all $F\in \mathbb{D}^{1,\,2}$.\\

A result from \cite{Nua} provides a more direct characterisation for a certain elementary subset of $\operatorname{Dom} (\delta )$: let $u=\sum_{i=1}^n F_i h_i$ be an element of $\mathcal{S}_{\mathcal{H}}$. Then $u$ belongs to $\operatorname{Dom} (\delta )$ and we have\
\begin{equation}
\label{eq:18}
\delta (u)=\sum_{i=1}^n F_i B^H (h_i)-\sum_{i=1}^n \langle DF_i,\, h_i\rangle_{\mathcal{H}}.
\end{equation}
We need the following identity for the expectation of $\delta (u)\delta(v)$, where $u$ and $v$ belong to the domain of $\delta$ and have a certain specific form. In \cite{Nua} a statement for more general $u$ and $v$ is given, however, in the present work we only need this special case: for $F_1,\,F_2\in \mathbb{D}^{1,\,2}$, $h_1,\,h_2\in\mathcal{H}$ the variables $u:=F_1 h_1$, $v:=F_2 h_2$ belong to the space $\operatorname{Dom} (\delta)$ and
\begin{equation}
\label{eq:19}
\E[\delta (u)\delta(v)]=\E[\langle u,\,v \rangle_{\mathcal{H}}]+ \E[\langle DF_1, h_2 \rangle_{\mathcal{H}}\langle DF_2, h_1 \rangle_{\mathcal{H}}].
\end{equation}

\vskip0.3cm

{\bf Acknowledgement:} The financial support of the Collaborative Research Center 'Statistical modeling of non-linear dynamic processes' (SFB 823) is gratefully acknowledged.

\end{document}